\documentclass[12pt, a4paper, reqno]{amsart}

\usepackage{hyperref}

\usepackage{graphicx}
\usepackage{color}
\usepackage{enumerate}
\usepackage{amssymb}
\usepackage{overpic}
\newtheorem{lemma}{Lemma}[section]
\newtheorem{corollary}[lemma]{Corollary}
\newtheorem{theorem}[lemma]{Theorem}
\newtheorem{proposition}[lemma]{Proposition}
\newtheorem{remark}[lemma]{Remark}

\author[A. Gasull]{A. Gasull$^*$}
\address{Dept. de Matem\`{a}tiques.
Universitat Aut\`{o}noma de Barcelona. Edifici C. 08193 Bellaterra, Barcelona. Spain}
\email{gasull@mat.uab.cat}

\author{H. Giacomini}
\address{Laboratoire de Math\'{e}matiques et Physique Th\'{e}orique. Facult\'{e} des
Sciences et Techniques. Universit\'{e} de Tours. 37200 Tours. France}
\email{Hector.Giacomini@lmpt.univ-tours.fr}

\author[J. Torregrosa]{J. Torregrosa$^*$}
\address{Dept. de Matem\`{a}tiques \\
Universitat Aut\`{o}noma de Barcelona \\ Edifici C. 08193 Bellaterra, Barcelona.
Spain} \email{torre@mat.uab.cat}

\thanks{$^*$\uppercase{T}he first and third authors are partially supported by
 the MCYT/FEDER grant number MTM2008-03437 by the CIRIT grant number 2009SGR410.}


\begin{document}

\title[Explicit bounds for the traveling wave solutions]
{Explicit upper and lower bounds\\ for the traveling wave solutions\\ of Fisher-Kolmogorov type equations}
\begin{abstract} It is well-known that the existence of traveling
wave solutions for reaction-diffusion partial differential equations
can be proved by showing the existence of certain heteroclinic orbits for related
autonomous planar differential equations. We introduce a method for finding
explicit upper and lower bounds of these heteroclinic orbits. In particular, for
the classical Fisher-Kolmogorov equation we give rational upper and lower bounds
which allow to locate these solutions analytically and with very high
accuracy.
\end{abstract}
\maketitle
\date{November 2011}

\section{Introduction and Main Results}\label{se:1}

Consider the adimensionalized reaction-diffusion partial differential equation
\begin{equation}\label{eq:1}
u_t=u_{xx}+ f(u),
\end{equation}
of Fisher-Kolmogorov type, where $f(u)$ is a smooth function satisfying certain
hypotheses.

The usual Fisher-Kolmogorov equation corresponds to $f(u)=u(1-u)$
and models the spreading of biological populations, see \cite{Fisher, Kolmo}.
Other well-known cases are the Newell-Whitehead-Segel equation, $f(u) = u(1-u^2),$
for describing Rayleigh-Benard convection, see \cite{NewWhi,Seg}, and the
Zeldovich equation, $f(u) = u(1 - u)(u - \alpha)$ with $0 < \alpha < 1,$
that appears in combustion theory, see \cite{ZelFra}. See also \cite{Gri,SanMai1995,SanMai1997,Xin}.

It is known that the traveling wave solutions
$u=u(x-ct)$ of \eqref{eq:1}, satisfying
\[
\lim_{s\to-\infty} u(s)=0\quad\mbox{and}\quad \lim_{s\to\infty} u(s)=1,
\]
appear when $u(s)$ is a special solution of the second order equation
\[
\ddot u+c\dot u+f(u)=0,
\]
where the dot indicates the derivative with respect to $s$. This
solution can be seen as the heteroclinic orbit $H_c$ of the planar system
\begin{equation}\label{eq:2}
\left\{
\begin{array}{ccl}
\dot u&=&v, \\
\dot v&=&-c v -f(u),
\end{array}
\right.
\end{equation}
that connects the origin with the saddle point at $(u,v)=(1,0)$, see
Figure~\ref{nfig1}. For instance for the classical Fisher-Kolmogorov equation it
exists only when $c\ge2.$ 

\begin{figure}[h]
\begin{center}
 \begin{overpic}{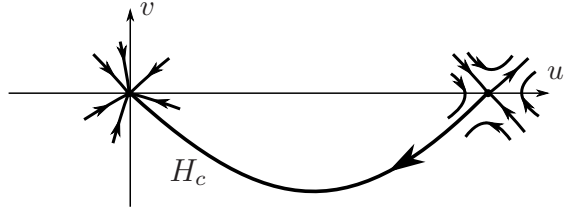}
 \put(30,5){$H_{c}$}
 \put(100,24){\small $u$}
 \put(24.5,36){\small $v$}
 \end{overpic}
\end{center}
\vspace{5mm} \caption{Phase portrait showing the unstable separatrix
$H_c$.}\label{nfig1}
\end{figure}

The goal of this paper is to give analytic upper and lower bounds of the heteroclinic orbit $H_c$ as well as of their time parametrization. We will approach to this question with similar tools to those introduced in \cite{GGT}. A key point consists in using the local behavior of the separatrices of the critical points to guess global algebraic bounds for the actual orbits.

First we prove a general result for system~\eqref{eq:2}. It is illustrated in Figure~\ref{nfig2}.

\begin{theorem}\label{thm:1.1} Consider system~\eqref{eq:2}, with $f$ satisfying $f(0)=f(1)=0,$ $f'(0)>0,$ $f'(1) <0,$ $f''(u)<0$ for all $u\in\mathbb{R}$ and $c\ge 2\sqrt{f'(0)}$. Let $H_c$ be its heteroclinic orbit and define
\[
\underline\lambda:=\dfrac{c-\sqrt{c^2-4f'(0)}}{2f'(0)}\quad\mbox{and}\quad
\overline\lambda:=\dfrac{c-\sqrt{c^2-4f'(1)}}{2f'(1)}.
\]
Then $H_c$ can be parametrized as $ H_c=\{(u,h_c(u)),\, u\in[0,1]\}$ and for
all $u\in(0,1),$
\[
-\underline{\lambda}\, f(u)< h_c(u)<-\overline {\lambda}\, f(u)<0.
\]
Moreover, if
 $(u_c(s),v_c(s))$ is the parametrization of $H_c$
such that $u_c(0)=1/2$ it holds that
\[
u_c(s)\in\left\langle z_{\underline \lambda}(s),
z_{\overline\lambda}(s)\right\rangle,
\]
where $\langle a,b\rangle$ denotes the smallest closed interval containing $a$ and
$b$ and $u=z_\lambda(s)$ is the solution of the Cauchy problem
\begin{equation*}
\left\{
\begin{array}{l}
\dfrac{du}{ds}=-\lambda f(1-u), \\[0.2cm]
u(0)=1/2.
\end{array}
\right.
\end{equation*}
\end{theorem}

\begin{figure}[h]
\begin{center}
 \begin{tabular}{cc}
 \begin{overpic}[height=3cm]{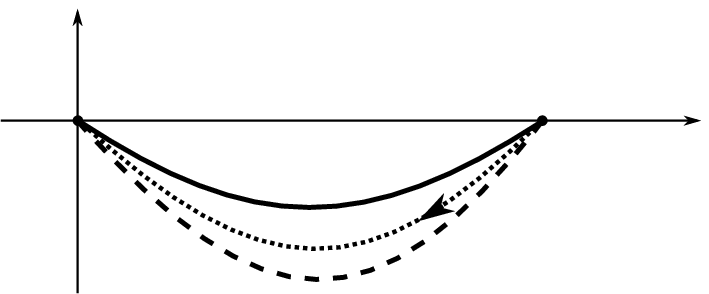}
 \put(33,18){\small $v=-\overline\lambda\, f(u)$}
 \put(62,3){\small $v=-\underline\lambda\, f(u)$}
 \put(94,27){\small$u$}
 \put(13,39){\small$v$}
 \put(30,-7){\small$(i)\quad \mbox{{\tiny Phase plane}}$}
 \end{overpic}
\vspace{0.3cm}&
 \begin{overpic}[height=3cm]{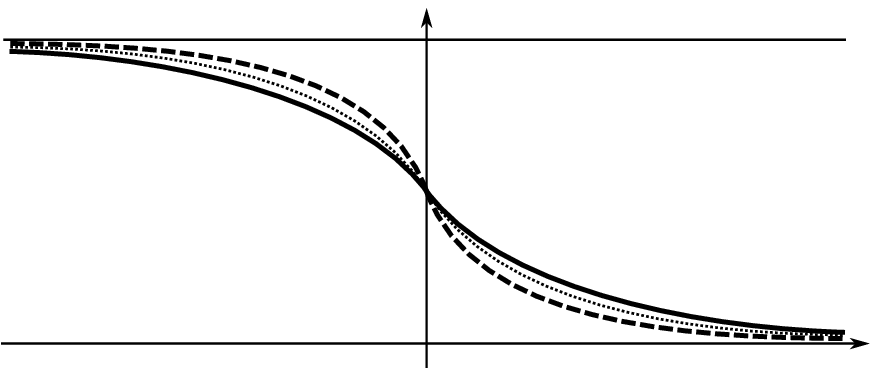}
 \put(35,18){\small$\left(0,\frac{1}{2}\right)$}
 \put(78,10){\small$u=z_{\overline\lambda}(s)$}
 \put(39.5,33){\small$u=z_{\underline\lambda}(s)$}
 \put(85,33.5){\small$u=1$}
 \put(95,-2){\small$s$}
 \put(51,40){\small$u$}
 \put(30,-7){\small$(ii)\quad \mbox{{\tiny Time parametrization}}$}
 \end{overpic}
 \end{tabular}
\end{center}
\caption{(i) Upper and lower bounds of the unstable separatrix $H_c$ (dotted line).
\,(ii) Bounds for the time parametrization of $H_c$.}\label{nfig2}
\end{figure}

For instance when we particularize the above theorem to the classical
Fisher-Kolmogorov system we obtain that
\begin{equation}\label{eq:3}
u_c(s)\in\left\langle \frac1{1+e^{\underline{\lambda}s}},
\frac1{1+e^{\overline{\lambda}s}}\right\rangle.
\end{equation}
The results obtained for system~\eqref{eq:2} can be improved when we study this case. We prove:

\begin{theorem}\label{thm:1.2} Let
$ H_c=\{(u,h^c(u)),\, u\in[0,1]\}$ be a parametrization of the heteroclinic
solution of system~\eqref{eq:2} when $f(u)=u(1-u)$. For $u\in(0,1)$ and $ c\ge2$
it holds that
\[
h^c_{2}(u)< h^c_{3}(u)<\cdots<
h^c_{100}(u)<h^c(u)<R^c_{10}(u)<R^c_9(u)<\cdots<R^c_1(u),
\]
where $h^c_n(u)$ is the Taylor polynomial of degree $n$ of the unstable separatrix
of the saddle point $(1,0)$ of the system and each $R^c_{m}$ is a rational
function whose numerator has degree $m+2$ and its denominator degree $m$,
constructed from the Pad\'{e} approximants of $v=h^c_{2m+2}(u)$. See
Figure~\ref{nfig3}.
\end{theorem}

\begin{figure}[h]
\begin{center}
 \begin{overpic}{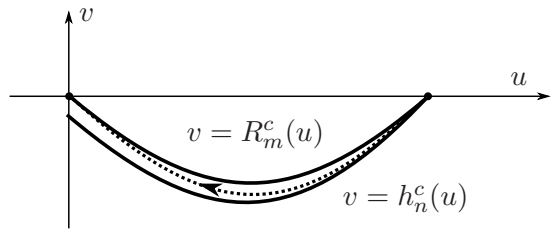}
 \put(33,16.5){\small $v=R_m^c(u)$}
 \put(62,4.5){\small $v=h_n^c(u)$}
 \put(93,27){\small $u$}
 \put(13,39){\small $v$}
 \end{overpic}
\end{center}
\vspace{5mm}
\caption{Upper and lower bounds of the separatrix $H_c$ of the Fisher-Kolmogorov equation, which is plotted as a dotted line.}\label{nfig3}
\end{figure}

For a given value of $c$ it is also possible to study the maximum distance between
two of the above functions. For instance when $c=99/100$, we can prove that
$R^{99/100}_9(x)-h^{99/100}_{20}(x)<2\times 10^{-19}.$ It is also possible to go
further in the computations. For example we get that
\[
h^{99/100}_{82}(u)< h^{99/100}(u)<R^{99/100}_{40}(u)\quad \mbox{and}\quad
R^{99/100}_{40}(u)-h^{99/100}_{82}(u)<2\times 10^{-47},
\]
for all $u\in[0,1]$ and the maximum error is at $u=0.$ For bigger values of $c$ it
is needed to compute approximations of higher degree to arrive to similar bounds
of the error.

The $s$-parametrization of $H_c$ can also be obtained with more accuracy for the
classical Fisher-Kolmogorov case. We only present here a first
result. Sharper approximations are detailed in Section~\ref{se:4}.

\begin{theorem}\label{thm:1.3}
Let $(u^c(s),v^c(s))$ be the time-parametrization of the heteroclinic orbit $H^c$
of system~\eqref{eq:2} with $f(u)=u(1-u)$ and such that $u^c(0)=1/2$. Define
\[
U^c(s)=\dfrac{3+2\sqrt2}{\left(1+\sqrt2+ e^{(\sqrt{c^2+4}-c)s/2}\right)^2}.
\]
Then:\begin{enumerate}[(i)]
\item When $2\le c<5/\sqrt{6}$ it holds that $\operatorname{sgn}(u^c(s)-U^c(s))=-\operatorname{sgn}(s).$
\item When $c=5/\sqrt{6}$ it holds that $u^c(s)=U^c(s).$
\item When $c>5/\sqrt{6}$ it holds that $\operatorname{sgn}(u^c(s)-U^c(s))=\operatorname{sgn}(s).$
\end{enumerate}
\end{theorem}

The above inequalities improve the bounds given in \eqref{eq:3}. Notice also that
when $c=5/\sqrt{6}$ we have obtained the exact expression of the traveling wave
solution of the classical Fisher-Kolmogorov equation, $u_t=u_{xx}+ u(u-1),$
\[
u(x,t)=\frac{3+2\sqrt{2}}{\left(1+\sqrt2+ e^{\frac{1}{\sqrt
6}\,\left(x-\frac5{\sqrt{6}}t\right) }\right)^2},
\]
which coincides with the one given in \cite{AZ}. The novelty of our result is that
a similar expression gives a bound of the traveling wave for all the values of
$c$.

The methods developed to study the classical Fisher-Kolmogorov equation can also be applied for the Newell-Whitehead-Segel and the Zeldovich equations.

\section{Proof of Theorem~\ref{thm:1.1}}\label{se:2}

For computational reasons it is more convenient to locate the saddle point of system~\eqref{eq:2} at the origin. So we introduce the new variables $x=1-u,$ $y=v$ and $t=s$ and it writes as
\begin{equation}\label{eq:4}
\left\{
\begin{array}{ccl}
 x'&=&-y, \\
y'&=&-c y +g(x),
\end{array}
\right.
\end{equation}
where $g(x)=-f(1-x)$ and the prime denotes derivative with respect to $t$. Observe that the variable $x$ introduced above does not coincide with the one used in equation~\eqref{eq:1}. Notice that $g$ satisfies the following set of hypotheses

\noindent ${\bf H:}$ $ g(0)=g(1)=0, g'(0)<0, g'(1)
>0 $ and $g''(x)>0$ for all $x\in\mathbb{R}$.

The above system has only two critical points $(0,0)$ and $(1,0)$.
The origin is a saddle point with eigenvalues
\begin{equation}\label{eq:5}
\lambda_s^{\pm}=\dfrac{-c\pm\sqrt{c^2-4g'(0)}}{2}
\end{equation}
and corresponding eigenvectors $(1,-\lambda_s^+)$ and $(1,-\lambda_s^-)$. Similarly, when $c^2-4g'(1)\ge0$ the point $(1,0)$ is an attracting node and its eigenvalues are
\begin{equation*}
\lambda_n^{\pm}=\dfrac{-c\pm\sqrt{c^2-4g'(1)}}{2}<0.
\end{equation*}
A sketch of the phase portrait of system~\eqref{eq:4} is given in Figure~\ref{nfig4}. There we can see the heteroclinic connection that we are interested
to locate. Indeed it is given by ones of the branches of the unstable separatrix
of the saddle point. In the new coordinates we will call it $\Gamma_c$.

\begin{figure}[h]
\begin{center}
 \begin{overpic}{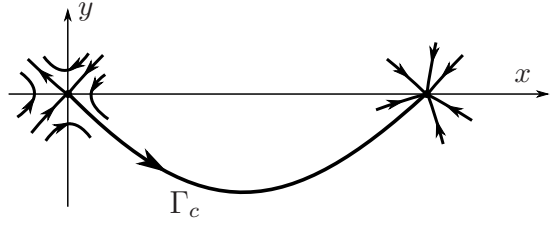}
 \put(30,-1){$\Gamma_{c}$}
 \put(94,23.5){\small $x$}
 \put(13,36){\small $y$}
 \end{overpic}
\end{center}
\vspace{5mm} \caption{Phase portrait showing the unstable separatrix $\Gamma_c$ of
system~\eqref{eq:4}.}\label{nfig4}
\end{figure}

\begin{proposition}\label{prop:2.1} Consider system~\eqref{eq:4}, with $g$
satisfying hypotheses ${\bf H}$ and $c\ge2\sqrt{g'(1)}$. Set
\[
\underline\lambda:=\dfrac{c-\sqrt{c^2-4g'(1)}}{2g'(1)}\quad\mbox{and}\quad
\overline\lambda:=\dfrac{c-\sqrt{c^2-4g'(0)}}{2g'(0)}.
\]
Let $\Gamma_c$ be its heteroclinic orbit. Then $\Gamma_c$ can be parametrized as
$ \Gamma_c=\{(x,\gamma_c(x)),\, x\in[0,1]\}$ and it holds for all $x\in(0,1)$
that
\[
\underline{\lambda}\, g(x)< \gamma_c(x)<\overline {\lambda}\, g(x)<0.
\]
\end{proposition}
\begin{proof} Consider the 1-parameter family of maps $G_\lambda(x,y)=y-\lambda g(x)$.
We compute
\begin{align*}
\left.\langle \nabla G_\lambda(x,y), \big(-y, -cy+g(x)\big)\rangle
\right|_{y=\lambda g(x)} =g(x)\big(1-c\lambda+g'(x)\lambda^2\big)=:g(x)N_\lambda(x).
\end{align*}
Since on $(0,1)$ it holds that $g(x)<0$, if we choose $\lambda$ such that
$N_\lambda$ does not
vanish on the same interval we will have that the corresponding curve $y=\lambda g(x)$ is without contact. Notice that $N_\lambda'(x)=\lambda^2 g''(x)$ and so for $\lambda\ne0$, the hypotheses ${\bf H}$ imply that the function $N_\lambda$ is increasing. Therefore:
 \begin{enumerate}[(a)]
\item If for some $\lambda>0$, $N_\lambda(0)\ge0$ then $N_\lambda(x)>0$ for all $x>0.$
\item If for some $\lambda>0$, $N_\lambda(1)\le0$ then $N_\lambda(x)<0$ for all $x<1$.
\end{enumerate}

The conditions $N_\lambda(j)=0$ for $j=0,1$, write as $1-c\lambda+g'(j)\lambda^2=0.$
Their solutions are
\begin{equation*}
\lambda_j^{\pm}=\dfrac{c\pm\sqrt{c^2-4g'(j)}}{2g'(j)}.
\end{equation*}
It is easy to prove that
\[\lambda_0^+<0<\lambda_0^-=:\overline {\lambda}< \underline{\lambda}:=
\lambda_1^-\le \lambda_1^+.\]

Therefore, taking $G_{\overline \lambda}(x,y)=0$ and $G_{\underline \lambda}(x,y)=0$ as an upper and lower boundary, respectively, we have constructed a subset of the strip $\{(x,y),\,0\le x\le1\}$, that contains the heteroclinic orbit $\Gamma_c$, see Figure~\ref{nfig5}. Then, on it, $x'=-y>0$, $\Gamma_c$ can be parametrized as a function of $x$, say $y=\gamma_c(x)$, and the inequalities of the statement follow.
\end{proof}

\begin{figure}[h]
\begin{center}
 \begin{overpic}{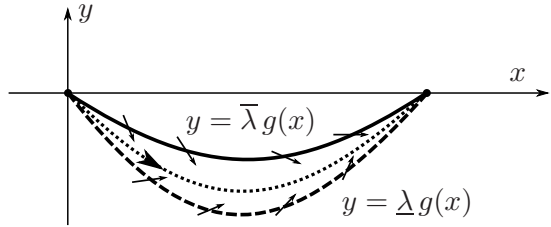}
 \put(33,18){\small $y=\overline\lambda\, g(x)$}
 \put(62,3){\small $y=\underline\lambda\, g(x)$}
 \put(93,27){\small $x$}
 \put(13,39){\small $y$}
 \end{overpic}
\end{center}
\vspace{5mm}
\caption{Upper and lower bounds of the unstable separatrix (dotted line) $\Gamma_c$.}\label{nfig5}
\end{figure}

Notice that $y=\overline{\lambda} g(x)=-\lambda_s^+ x+O(x^2)$ and therefore this curve is tangent to the unstable separatrix of the saddle point.

Using the above proposition we can also approach the parametrization of $\Gamma_c$ with respect the actual time $t.$ We need to introduce some new functions. Given the Cauchy problem
\begin{equation}\label{eq:6}
\left\{
\begin{array}{l}
\dfrac{dx}{dt}=-\lambda g(x), \\[0.2cm]
x(0)=1/2,
\end{array}
\right.
\end{equation}
we will denote by $x=w_\lambda(t)$ its corresponding solution.

\begin{theorem}\label{thm:2.2}
Under the same hypotheses and notations of Proposition~\ref{prop:2.1}, if
$(x_c(t),y_c(t))$ is the parametrization of $\Gamma_c$ such that $x_c(0)=1/2$ it holds that
\begin{equation*}
x_c(t)\in\left\langle
w_{\underline\lambda}(t),w_{\overline\lambda}(t)\right\rangle,
\end{equation*}
where $w_{\lambda}(t)$ is defined in \eqref{eq:6}, see Figure~\ref{nfig6}.
\end{theorem}

\begin{figure}[h]
\begin{center}
 \begin{overpic}{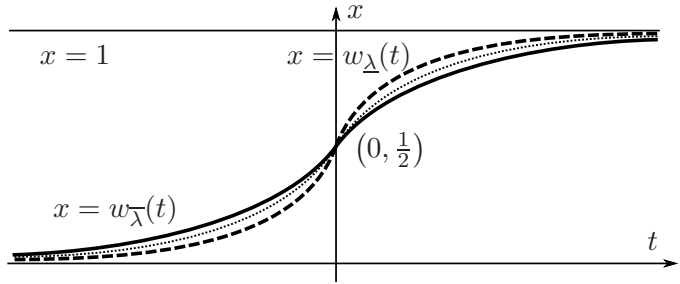}
 \put(52,19){\small $\left(0,\frac{1}{2}\right)$}
 \put(42,33){\small $x=w_{\underline\lambda}(t)$}
 \put(7,10){\small $x=w_{\overline\lambda}(t)$}
 \put(96,5){\small $t$}
 \put(5,33){\small $x=1$}
 \put(51,40){\small $x$}
 \end{overpic}
\end{center}
\vspace{5mm}
\caption{Upper and lower bounds of the $t$-parametrization of $\Gamma_c$ (dotted curve).}\label{nfig6}
\end{figure}

\begin{proof} By Proposition~\ref{prop:2.1} we know that
\[
\underline{\lambda}\, g(x)< \gamma_c(x)<\overline {\lambda}\, g(x)<0.
\]
Recall that $\Gamma_c$ is parametrized as $y=\gamma_c(x)$ and therefore since
in system~\eqref{eq:4},
$x'(t)=-y(t)$, it holds that $x_c'(t)=-\gamma_c(x_c(t))$. Therefore
\[
-\underline{\lambda}\, g(x_c(t))> x_c'(t)>-\overline {\lambda}\, g(x_c(t))>0.
\]
Since $x_c(0)=w_{\overline {\lambda}}(0)=1/2$, $x_c(t)$ satisfies the differential inequality
\[
x_c'(t)>-\overline {\lambda}\, g(x_c(t))
\]
and $w_{\overline {\lambda}}(t)$ the equality, it holds that
$x_c(t)>w_{\overline {\lambda}}(t)$ for all $t>0,$ as we wanted to see.
The other cases follow similarly.
\end{proof}

\begin{corollary}\label{cor:2.3} If in system~\eqref{eq:4},
$g(x)=x^m-x,$ $2\le m\in\mathbb{N},$ then
\[
x_c(t)\in\left\langle\frac{1}{\sqrt[m-1]{1+(2^{m-1}-1)
e^{\phantom{\overline{\lambda}}\!\!\!-(m-1)\underline{\lambda}\,t}}},
\frac{1}{\sqrt[m-1]{1+(2^{m-1}-1)e^{-(m-1)\overline{\lambda}\,t}}}\right\rangle.
\]
 In particular for the Fisher-Kolmogorov case, $m=2$,
\begin{equation}\label{eq:7}
x_c(t)\in\left\langle\frac{1}{{1+e^{-\underline{\lambda}\,t}}},
\frac{1}{{1+e^{-\overline{\lambda}\,t}}}\right\rangle.
\end{equation}
\end{corollary}
\begin{proof} It suffices to solve the Cauchy problem~\eqref{eq:6} and then apply
Theorem~\ref{thm:2.2}. When $g(x)=x^m-x$ we obtain
\[
w_\lambda(t)=\frac{1}{\sqrt[m-1]{1+(2^{m-1}-1)e^{-(m-1){\lambda}t}}}
\]
and so the result follows.
\end{proof}

Theorem~\ref{thm:1.1} is a reformulation for system~\eqref{eq:2} of the results of
this section obtained for system~\eqref{eq:4}.

\section{Sharper upper and lower bounds for $\Gamma_c$ in the Fisher-Kolmogorov case}\label{se:3}

This section will be devoted to find sharper upper and lower bounds for
$\Gamma_c$ in the
Fisher-Kolmogov system~\eqref{eq:4} when $g(x)=x(x-1)$.
We will use dynamical tools inspired in \cite{GGT}. One of the key points will be to find algebraic curves constructed by imposing that these curves coincide as much as possible with the unstable separatrix of the saddle point.

To avoid the appearance of square roots during the computations it
is convenient to include a new parameter $r$ in such a way that
\[
c=\dfrac 1 r-r.
\]
Then system~\eqref{eq:4} writes as
\begin{equation}\label{eq:8}
\left\{
\begin{array}{ccl}
 x'&=&-y, \\
y'&=&\left(r-\dfrac1r\right)y +x(x-1),\quad r\in(0,\sqrt2-1].
\end{array}
\right.
\end{equation}
Notice that the condition on $r$ implies that $c=r-1/r\ge 2$. One advantage of
introducing this new parameter is that the eigenvalues of the saddle are now
$-1/r<0<r$. In the notation of the previous section
$\lambda_s^+=\overline\lambda=r.$ Hence, if we denote $\Gamma^r:=\Gamma_{1/r-r}$
the searched heteroclinic trajectory, from Proposition~\ref{prop:2.1} we know that
$y=rx(x-1)$ is an upper bound for $\Gamma^r.$

First we need to know the local expansion of the unstable manifold of the
saddle point.

\begin{lemma}\label{lem:3.1}
The local unstable manifold of the origin of system~\eqref{eq:8} writes as the analytic function
\begin{align*}
y=h^r(x)=&-rx+\dfrac r{2r^2+1}x^2+\dfrac{2r^3}{(2r^2+1)^2(3r^2+1)}x^3\\&+
\dfrac{10r^5}{(2r^2+1)^3(3r^2+1)(4r^2+1)}x^4\\&
+\dfrac{12r^7(19r^2+6)}{(2r^2+1)^4(3r^2+1)^2(4r^2+1)(5r^2+1)}x^5 + O(x^5),
\end{align*}
and the subsequent terms can be determined recurrently and are positive for $r>0$.
\end{lemma}
\begin{proof}
Let $y=h^r(x)=h(x)$ be the local expression of any of the separatrices of the
saddle point, being $h$ an analytic function at zero. Then
$\left.y'-h'(x)x'\right|_{y=h(x)}\equiv0,$ or equivalently,
\begin{equation}\label{eq:9}
\left.\left(r-\dfrac 1r\right)y+x(x-1)+h'(x)y\right|_{y=h(x)}=
\left(r-\frac1r+h'(x)\right)h(x)-x+x^2\equiv0.
\end{equation}
Writing $h(x)=h_1x+h_2x^2+\ldots$ and plugging this expression in the above one we
get that $h_1^2+(r-1/r)h_1-1=0.$ So we choose $h_1=-r.$ Then the right hand
identity in~\eqref{eq:9} writes as
\[
\left(-\frac 1r +2h_2x+3h_3x^2+\cdots\right)(-rx+h_2x^2+h_3x^3+\cdots)-x+x^2\equiv0.
\]
So $-2h_2rx^2-h_2x^2/r+x^2\equiv0$, which implies $h_2=r/(2r^2+1).$ In general,
for $n>2$ it holds that
\[
-\dfrac 1rh_n+2h_2h_{n-1}+3h_3h_{n-2}+\cdots+(n-1)h_{n-1}h_2 -nh_nr=0.
\]
Therefore,
\[
h_n=\frac{2h_2h_{n-1}+3h_3h_{n-2}+\cdots+(n-1)h_{n-1}h_2}{nr+1/r}
\]
and by induction $h_n>0$, for $n>2$, as we wanted to prove.
\end{proof}

\begin{figure}[h]
\begin{center}
 \begin{overpic}{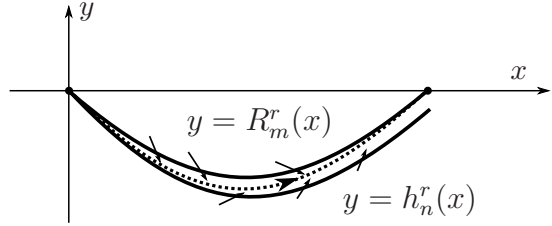}
 \put(33,17.5){$y=R_m^r(x)$}
 \put(62,3){$y=h_n^r(x)$}
 \put(93,27){\small $x$}
 \put(13,39){\small $y$}
 \end{overpic}
\end{center}
\vspace{5mm}
\caption{Upper and lower bounds of the separatrix $\Gamma_c$ plotted as a dotted thin line. Here $2\le n\le 100$ and $1\le m \le 10$.}\label{nfig7}
\end{figure}

\begin{proposition}\label{prop:3.2} Let
$ \Gamma^r=\{(x,\gamma^r(x)),\, x\in[0,1]\}, $ be a parametrization of the
heteroclinic solution of system~\eqref{eq:8}. Then for all $x\in(0,1)$ it holds
that
\[
h^r_{2}(x)< h^r_{3}(x)<\cdots< h^r_{99}(x)< h^r_{100}(x)<\gamma^r(x),
\]
where $h^r_n(x)$ is the Taylor polynomial of degree $n$ in powers of $x$ of the
function $h^r(x)$ defined in Lemma~\ref{lem:3.1}. See Figure~\ref{nfig7}.
\end{proposition}

\begin{proof} We will give first the details for $n=3$. It is evident that for $x>0$ sufficiently small we
have $h^r(x)>h^r_3(x)$.
First we evaluate the polynomial $h^r_3(x)$ at $x=1$. We obtain
\[
h^r_3(1)=-\dfrac{2r^5(5+6r^2)}{(1+2r^2)^2(1+3r^2)}.
\]
This quantity is negative for $r>0$ and it is a monotonous decreasing function of
$r$. Its value at $r=\sqrt{2}-1$ is approximately $-0.054$. Moreover,
$h^r_3(0)=0$, $h^r_3(x)''>0$ for $x>0$ and $h^r_3(x)<0$ for $x>0$ and sufficiently
small. Then we conclude that $h^r_3(x)<0$ for all $x\in(0,1]$. We will show now
that the flow of the vector field $Z^r(x,y)=(-y, (r-1/r)y +x(x-1))$ associated to system~\eqref{eq:8} crosses the curve $F^r_3(x,y)=y-h^r_3(x)=0$ upwards. We
compute
\[
M^r_3(x):=\left.\left\langle \nabla F^r_3(x,y), Z^r(x,y)
\right\rangle\right|_{y=h^r_3(x)}.
\]
We obtain
\[
M^r_3(x)=\dfrac{10r^4x^4}{(1+2r^2)^3(1+3r^2)}+\dfrac{12r^6x^5}{(1+2r^2)^4(1+3r^2)^2}.
\]
This expression is positive for $x>0$.

All the other cases can be studied by using the same method. The key point is that
all the monomials in the corresponding expression of $M_n^r(x)$ are positive for $x>0.$
\end{proof}

In Proposition~\ref{prop:3.2}, we have chosen to stop at the value 100 because the
computation of the function $M^r_{100}$ used in its proof together with the
testing that all its coefficients are positive takes more than four hours of CPU
time in our computer. In any case, for practical uses, it suffices to consider small values of $n$.

\begin{remark}\label{rem:3.3}
It is clear that if we could prove that the radius of convergence of the series given in Lemma~\ref{lem:3.1}, when $r\in(0,\sqrt{2}-1]$, is 1 we would have obtained an infinite monotonous sequence of polynomials tending to the actual separatrix. Unfortunately we have not succeeded in our attempts. In any case, computing an approximation of the radius of convergence, by using several hundreds of terms of the series, seems to show that the result is true.
\end{remark}

We give now upper bounds for $\Gamma^r$. Before to state the next
proposition we introduce some definitions. Consider the rational function
\[
R^r(x)=\dfrac{h^r_{22}(x)}{rx(x-1)}.
\]
We define the sequence of rational functions $R^r_n(x)$ as follows:
\[
R^r_n(x)=rx(x-1)P^r_{n,n}(x),\quad n=1,\ldots,10,
\]
where the $P^r_{n,n}(x)$ are the Pad\'e approximants of order $(n,n)$ of the
function $R^r(x)$. For instance
\[
R^r_1(x)=rx(x-1)\frac{(2r^2+1)(3r^2+1)-3r^2x}{(2r^2+1)(3r^2+1)-r^2(6r^2+5)x}.
\]
Recall that the Pad\'e approximants $P^r_{n,n}(x)$ are rational functions whose
numerators and denominators are polynomials of degree $n$ and their Taylor
expansions in powers of $x$ are the same that the Taylor expansion of $R^r(x)$ up
to order $2n$. We write
\[R_n^r(x)=rx(x-1)\frac{A_n^r(x)}{C_n^r(x)}, \quad \mbox{with}\quad C_n^r(0)>0.\]
The rational functions $R^r_n(x)$ vanish at $x=0$ and $x=1$ and their Taylor
expansions coincide with the Taylor expansion of $h^r(x)$ up to order $2n$. As we
will see in the proof of next proposition they are well defined for all
$x\in[0,1]$ because $C_n^r(x)\ne0$.

\begin{proposition}\label{prop:3.4} Let
$ \Gamma^r=\{(x,\gamma^r(x)),\, x\in[0,1]\}, $ be a parametrization of the
heteroclinic solution of system~\eqref{eq:8}. For $x\in(0,1)$ and $r\in(0,\sqrt{2}-1]$ it holds that
\[
\gamma^r(x)<R^r_{10}(x)<R^r_9(x)<\cdots<R^r_2(x)<R^r_1(x),
\]
see Figure~\ref{nfig7}.
\end{proposition}

\begin{proof}
We will show that for each $n,$ $1\le n\le 10,$ the flow of the vector field
$Z^r(x,y)=(-y, (r-1/r)y +x(x-1))$ associated to system~\eqref{eq:8} crosses
each curve $G^r_n(x,y)=y-R^r_n(x)=0$ forwards. We compute
\[
N^r_n(x):=\left.\left\langle \nabla G^r_n(x,y), Z^r(x,y)
\right\rangle\right|_{y=R^r_n(x)}.
\]
We obtain
\[
N^r_n(x)=r^{4n+2}x^{2n+2}(x-1)\dfrac{B^r_n(x)}{C^r_n(x)^3},
\]
where $B^r_n(x)$ and $C^r_n(x)$ are polynomials of degree $n$ in $x$. The
coefficients of $B^r_n(x)$ and $C^r_n(x)$ are polynomials in $r$.

We want to prove that for $r>0$ both polynomials $B^r_n(x)$ and $C^r_n(x)$ are positive for $x\in(0,1)$. We will approach the problem varying the parameter $r$ and studying
which are the possible bifurcations for the number of zeros of these polynomials
when $x\in(0,1)$. The values of $B^r_n(0)$, $C^r_n(0)$, $B^r_n(1)$ and
$C^r_n(1)$, are polynomials in $r$ whose coefficients are natural numbers. Then we
have $B^r_n(0)>0$, $C^r_n(0)>0$, $B^r_n(1)>0$ and $C^r_n(1)>0$ for $r>0$. For
$n\ge2$, we evaluate now the discriminants of the polynomials $B^r_n(x)$ and
$C^r_n(x)$, $\operatorname{Dis}(B^r_n(x),x)$ and $\operatorname{Dis}(C^r_n(x),x)$.
We obtain a polynomial in $r$ of degree $b_n$ with positive integers numbers as
coefficients for $\operatorname{Dis}(B^r_n(x),x)$ and a polynomial in $r$ of
degree $c_n$ with also positive integer numbers as coefficients for
$\operatorname{Dis}(C_n^r(x),x)$, where the degrees $b_n$ and $c_n$ are given in
Table~\ref{ta:1}. Then, the two discriminants do not vanish for $r>0$. In consequence, the
number of real roots in $(0,1)$ of each polynomial, $B^r_n(x)$ or $C^r_n(x)$,
does not change for $r>0$. Picking a concrete value of $r$, for instance
$r=1/10$, we found, by applying the Sturm algorithm, that they have no roots in
$(0,1)$. Then, we deduce that $B^r_n(x)>0$ and $C^r_n(x)>0$ for $x\in(0,1)$ and
$r>0$. Therefore, we have proved that for $n\ge2,$ $N^r_n(x)<0$ for $x\in(0,1)$ and $r>0$, as we wanted to see. The case $n=1$ is much easier because
we do not need to compute the discriminants and we omit the details.

\begin{table}[h]
\begin{tabular}{|c|c|c|c|c|c|c|c|c|c|}
\hline
$n$ & 2 & 3 & 4 & 5 & 6 & 7 & 8 & 9 & 10 \\
\hline
$b_n$ & 40 & 212 & 624 & 1480 & 2900 & 5028 & 8260 & 12560 & 18180 \\
$c_n$ & 24 & 100 & 264 & 584 & 1100 & 1860 & 2996 & 4496 & 6444 \\
\hline
\end{tabular}
\vspace{0.3 cm}
\begin{caption}
{Degrees of the numerator and the denominator of the resultants computed in the proof of Proposition~\ref{prop:3.4}.}
\end{caption}\label{ta:1}
\end{table}
Finally for $k=2,3,\ldots,10$ we obtain by a direct computation
\[
R^r_{k}(x)-R^r_{k-1}(x)=\frac{x^{2k}(x-1)\,r^{4k-1}D_k(r^2)}{C^r_k(x)C_{k-1}^r(x)},
\]
where $D_k(x)$ is a polynomial with positive coefficients. Hence
$R^r_{k}(x)<R^r_{k-1}(x)$ as we wanted to show.
\end{proof}

As a corollary of Propositions~\ref{prop:3.2} and \ref{prop:3.4} we obtain the following
result:
\begin{theorem}\label{thm:3.5}
Let $\Gamma^r=\{(x,\gamma^r(x)),\, x\in[0,1]\}, $ be a parametrization of the
heteroclinic solution of system~\eqref{eq:8}. For $x\in(0,1)$ and $r\in(0,\sqrt{2}-1]$ it holds that
\[
h^r_{2}(x)< h^r_{3}(x)<\cdots<
h^r_{100}(x)<\gamma^r(x)<R^r_{10}(x)<R^r_9(x)<\cdots<R^r_1(x).
\]
\end{theorem}

If we were interested in obtaining more precise upper
bounds for the heteroclinic orbit, we could simply increase $n$ and apply the same
procedure. For instance for $r=1/10$ we have performed all the computations and
proved that for $x\in(0,1)$,
\begin{align*}
&h^{1/10}_{42}(x)< \gamma^{1/10}(x)<R^{1/10}_{20}(x)&\mbox{and}\qquad
R^{1/10}_{20}(x)-h^{1/10}_{42}(x)<9\times 10^{-32},\\
&h^{1/10}_{62}(x)< \gamma^{1/10}(x)<R^{1/10}_{30}(x)&\mbox{and}\qquad
R^{1/10}_{30}(x)-h^{1/10}_{62}(x)<2\times 10^{-40},\\
&h^{1/10}_{82}(x)< \gamma^{1/10}(x)<R^{1/10}_{40}(x)&\mbox{and}\qquad
R^{1/10}_{40}(x)-h^{1/10}_{82}(x)<2\times 10^{-47},
\end{align*}
where the functions $R_m^r(x),$ $m>10,$ are defined similarly to the ones given in
the above theorem. We remark that in all the cases the maximum error is at $x=1.$ It is also important to notice that for bigger values of $r$ we need bigger values of $n$ in $R^r_n(x)$ and $h^r_n(x)$ to obtain similar bounds.

Theorem~\ref{thm:1.2} is simply a reformulation of the above theorem.

\section{On the time-parametrization of $\Gamma_c$ for the\\ Fisher-Kolmogorov
case}\label{se:4}

By using normal forms theory it is well-known that in a neighborhood of the
origin of system~\eqref{eq:8} its unstable manifold can be parametrized as
$F(e^{rt})$ for some analytic function $F$. So it seems natural to find bounds of the actual heteroclinic orbit which are rational functions of $e^{rt}$. As far as we know this idea is new.

First, we consider the family of rational functions
\[
X(t)=\dfrac{\beta e^{rt} +\alpha e^{2rt}}{1+(\alpha+2\beta-1)e^{rt}+\alpha e^{2rt}}
\]
to try to approximate the function $x^r(t)$, where we denote by $(x^r(t),y^r(t))$
the time-parametrization of the heteroclinic orbit $\Gamma^r$. Notice that the
parameters are taken in such a way that
\[
\lim_{t\to-\infty} X(t)=0,\quad X(0)=\dfrac12,\quad \lim_{t\to\infty} X(t)=1,
\]
properties that are also satisfied by $x^r(t)$. We do not care about the $y$
component because $y(t)=-x'(t)$.

Recall that the vector field associated to system~\eqref{eq:8} is
\[Z^r(x,y)=(Z^r_1(x,y),Z^r_2(x,y))=(-y, (r-1/r)y +x(x-1)).\]
To study the behavior of the flow of this system on the curve $(X(t),Y(t)),$
where $Y(t)=-X'(t)$ we compute
\begin{equation}\label{eq:10}
M(t):=-Y'(t)Z^r_1(X(t),Y(t))+X'(t)Z^r_2(X(t),Y(t)).
\end{equation}
If we introduce the compact notation $\Phi:=e^{rt}$ we obtain that
\[
M(t)=\dfrac{\Phi^3\left(\alpha(\alpha+\beta-1)\Phi^2+2\alpha
\Phi+\beta\right)P_3(\Phi)} {(\alpha\Phi^2+(\alpha+2\beta-1)\Phi+1)^5},
\]
where $P_3$ is a polynomial of degree 3 with coefficients depending also
polynomially on $\alpha$ and $\beta$.

In order to simplify the expression of $M(t)$ we take $\alpha=1-\beta$. Then
\[
M(t)=\dfrac{\Phi^3\left(2(1-\beta)\Phi+\beta\right)P_3(\Phi)}
{((1-\beta)\Phi^2+\beta\Phi+1)^5}.
\]
Finally, taking $\beta=2\sqrt2-2,$ the numerator and the denominator of the above
fraction have a common zero. Then it writes as
\[
M(t)=\dfrac{2(17+12\sqrt2)r(1-6r^2)\Phi^3}{(1+\sqrt2+\Phi)^7}.
\]
In short we have proved the following:

\begin{theorem}\label{thm:4.1}
Let $(x^r(t),y^r(t))$ be the time-parametrization of the heteroclinic orbit
$\Gamma^r$ of system~\eqref{eq:8} that satisfies $x^r(0)=1/2$.
Define
\[
X^r(t)=\dfrac{(2+2\sqrt2+ e^{rt})e^{rt}}{(1+\sqrt2+ e^{rt})^2}.
\]
Then:\begin{enumerate}[(i)]
\item When $r<1/\sqrt{6}$ it holds that
$\operatorname{sgn}(x^r(t)-X^r(t))=-\operatorname{sgn}(t).$
\item When $r=1/\sqrt{6}$ it holds that $x^r(t)=X^r(t).$
\item When $r>1/\sqrt{6}$ it holds that
$\operatorname{sgn}(x^r(t)-X^r(t))=\operatorname{sgn}(t).$
\end{enumerate}
\end{theorem}

Theorem~\ref{thm:1.3} follows from the above result, simply using that $u=1-x$ and
$c=1/r-r$.

Theorem~\ref{thm:3.5} can also be used to obtain an explicit
bound for $x^r(t).$ Since $h_2(x)=-rx+\dfrac r{2r^2+1}x^2<\gamma^r(x)$, where
$(x,\gamma^r(x))$ is the parametrization of $\Gamma^r$, solving the Cauchy
problem
\begin{equation*}
x'=rx-\dfrac r{2r^2+1}x^2, \quad x(0)=\frac12,
\end{equation*}
we obtain a function
\[
U^r(t)=\frac{2r^2+1}{1+(4r^2+1)e^{-rt}},
\]
such that $\operatorname{sgn}(x^r(t)-U^r(t))=-\operatorname{sgn}(t).$
The other bounds given in Theorem~\ref{thm:3.5} give rise to
implicit bounds of the form $H(t,x)=0$ of the curve $(t,x^r(t))$.

\subsection{Sharper bounds when $\boldsymbol{r\in(0,1/\sqrt{6})}$}\label{se:4.1}
This subsection improves the results of the previous section when $r<1/\sqrt{6}$.

Fixed $r\in(0,1/\sqrt{6})$ and given any natural $n$, following similar techniques
that the ones used to prove Lemma~\ref{lem:3.1}, we can compute a function
$x^r_{n}(t)=\sum_{j=1}^{2n}a_j \Phi^j$, with $\Phi=e^{rt}$ and $a_1=1$, that
coincides with the solution $x^r(t)$ until order $2n$ in $\Phi$. For instance
\[
a_2= -\frac{1}{2r^2+1} \quad\mbox{and}\quad a_3= \frac{1}{(2r^2+1)(3r^2+1)}.
\]
Looking $\Phi$ as an independent variable, we compute the associated Pad\'e
approximant of $\sum_{j=1}^{2n}a_j \Phi^j$ of order $(n,n)$, obtaining
\[
Z_n^r(t)=\widetilde{Z}_n^r(\Phi)=\frac{\sum\limits_{j=1}^{n} b_j(r)
\Phi^{j}}{\sum\limits_{j=0}^{n} c_j(r) \Phi^{j}},
\]
where $b_j(r)$ and $c_j(r)$ are polynomials on $r.$ For instance
\[
\widetilde{Z}_2^r(\Phi)=\frac{3 (2 r^2+1) (3 r^2+1) (4 r^2+1) \Phi-2 (r-1) (r+1)
(3 r^2+1) \Phi^2}{3 (2 r^2+1) (3 r^2+1) (4 r^2+1)+5 (2 r^2+1) (3 r^2+1) \Phi+(3
r^2+2) \Phi^2}.
\]

Notice that it satisfies that $\lim\limits_{t\rightarrow -\infty} Z_n^r(t)=0$ but
we do not impose neither that $Z_n^r(0)=1/2$ nor that $\lim\limits_{t\rightarrow
\infty} Z_n^r(t)=1.$

As in the proof of Theorem~\ref{thm:4.1}, to study the behavior of the flow
associated to system~\eqref{eq:8} on the curve $(Z_n^r(t), -Z_n^r(t)')$ we
compute the corresponding function~\eqref{eq:10}. It writes as
\[
M_n(t)=\widetilde{M}_n(\Phi)=\dfrac{\Phi^{2n+2}r^3(1-6r^2)P_{3n-3}(\Phi)
}{(Q_{n}(\Phi))^5},
\]
where $P_k$ and $Q_k$ are polynomials of degree $k$ with coefficients depending
also polynomially on $r$. It can be seen that all these coefficients,
as functions of $r$, take positive values on $(0,1/\sqrt{6}).$ We prove this fact
introducing a new variable $z$ satisfying $r=z^2/(\sqrt{6}(1+z^2))$ in each of the
coefficients and then applying the Descartes' Theorem to the resulting
polynomials. Then $M_n(t)>0$ for all $t$ and for all $r$ in $(0,1/\sqrt{6}).$
Moreover, this also shows that $Z_n^r(t)$ is well defined because
$Q_n(\Phi)=\sum\limits_{j=0}^{n} c_j(r) \Phi^{j}>0.$

Since we want that $Z_n^r(0)=1/2$ we need to modify $Z_n^r(t).$ To do this, for
$\rho>0,$ we define the new family of functions
$W_n^r(t,\rho)=\widetilde{Z}_n^r(\rho\Phi).$ It is not difficult to see that
following the above procedure the same results hold. Moreover it can be proved
that there exists a unique value $\rho_0(r)\in(0,1)$ such that
$W_n^r(0,\rho_0(r))=1/2.$ So we define $X_n^r(t)=W_n^r(t,\rho_0(r))$ and it
satisfies $X_n^r(0)=1/2.$ Since the corresponding function~\eqref{eq:10} is positive,
for each $n=2,\ldots,8,$ it holds that
\[
\operatorname{sgn}(x^r(t)-X^r(t))=-\operatorname{sgn}(t).
\]
Moreover $\lim\limits_{t\rightarrow \infty} X_n^r(t)=b_n(r)/c_n(r)$ and it can be
seen that these limits, which are rational functions in $r$, satisfy
\[
1< b_8(r)/c_8(r)<b_7(r)/c_7(r)<\ldots <b_2(r)/c_2(r).
\]
Furthermore
\[
E_k:=\max_{r\in(0,1/\sqrt{6})}\left(\frac{b_k(r)}{c_k(r)}-1\right)
\]
decreases with $k$. For instance $E_2=3-4\sqrt{5}/3\approx 0.02$ and $E_8\approx
0.0007$.

\section*{Conclusions} In spite of the great interest of studying the
traveling wave solutions of reaction-diffusion equations, $u_t=u_{xx}+ f(u),$
there are no methods for obtaining explicit bounds for them. In this paper we
present an approach that allows to obtain this type of bounds in the general case.
Moreover we introduce more elaborated tools, based on the control of the
heteroclinic trajectories of an associated planar ordinary differential equation,
that allow to improve the general results when we deal with a particular function
$f$. We study with detail the classical Fisher-Kolmogorov case $f(u)=u(1-u).$ The
methods developed can be easily adapted to treat the Newell-Whitehead-Segel and
the Zeldovich equations

\end{document}